\documentclass[12pt]{article}

\usepackage{diagbox}
\usepackage{mathtools}
\usepackage{bbm}
\usepackage{latexsym}
\usepackage{epsfig}
\usepackage{amsmath,amsthm,amssymb,enumerate,hyperref}
\usepackage[active]{srcltx}
\usepackage{color}

\def\hw{\widehat{w}}

\newcounter{rot}

\newcommand{\hatt}[1]{\widehat #1}

\def\a{\alpha} \def\b{\beta}  \def\D{\Upsilon}
\def\e{\varepsilon}    \def\g{\gamma}
\def\G{\Gamma}  
     \def\l{\lambda}
  \def\n{\nu} 
  
 \def\om{\omega}

\def\hY{\hatt{Y}}
\def\hX{\hatt{X}}

\def\cP{{\cal P}}

\def\cT{{\cal T}}
\def\cM{{\cal M}}

\newtheorem{theorem}{Theorem}
\newtheorem{lemma}[theorem]{Lemma}
\newtheorem{corollary}[theorem]{Corollary}

\def\cT{{\mathcal T}}

\newcommand{\brac}[1]{\left(#1\right)}

\newcommand{\bfrac}[2]{\left(\frac{#1}{#2}\right)}

\def\cE{{\cal E}}

\newcommand{\set}[1]{\left\{#1\right\}}

\def\Var{{\bf Var}}

\newcommand{\ignore}[1]{}

\newcommand{\beq}[2]{\begin{equation}\label{#1}#2\end{equation}}
\def\bc{{\bf c}}
\def\bC{{\bf C}}

\newcommand{\multstar}[1]{\begin{multline*}#1\end{multline*}}
\newcommand{\mult}[2]{\begin{multline}\label{#1}#2\end{multline}}
\usepackage{tikz}
\usetikzlibrary{arrows}
\usetikzlibrary{decorations}
\usetikzlibrary{shapes.misc}

\newcommand{\dd}{\mathrm{d}}
 \newcommand{\E}{\mathbb{E}}
\def\Pr{\mathbb{P}}

\def\cS{{\mathcal S}}
\begin{document}
\author{Alan Frieze\thanks{Department of Mathematical Sciences,Carnegie Mellon University,Pittsburgh PA 15213, Research supported in part by NSF grant DMS1661063}, Wesley Pegden\thanks{Department of Mathematical Sciences,Carnegie Mellon University,Pittsburgh PA 15213. Research supported in part by NSF grant DMS}, Gregory Sorkin\thanks{Department of Mathematics, London School of Economics, London,}, Tomasz Tkocz\thanks{Department of Mathematical Sciences,Carnegie Mellon University,Pittsburgh PA 15213. Research supported in part by the Collaboration Grants from the Simons Foundation.}}

\title{Minimum-weight combinatorial structures under random cost-constraints}
\maketitle

 \begin{abstract}
 Recall that Janson showed that if the edges of the complete graph $K_n$ are assigned exponentially distributed independent random weights, then the expected length of a shortest path between a fixed pair of vertices is asymptotically equal to $(\log n)/n$.  We consider analogous problems where edges have not only a random length but also a random cost, and we are interested in the length of the minimum-length structure whose total cost is less than some cost budget.  For several classes of structures, we determine the correct minimum length structure as a function of the cost-budget, up to constant factors.  Moreover, we achieve this even in the more general setting where the distribution of weights and costs are arbitrary, so long as the density $f(x)$ as $x\to 0$ behaves like $cx^\gamma$ for some $\gamma\geq 0$; previously, this case was not understood even in the absence of cost constraints.  We also handle the case where each edge has several independent costs associated to it, and we must simultaneously satisfy budgets on each cost.  In this case, we show that the minimum-length structure obtainable is essentially controlled by the product of the cost thresholds.
 \end{abstract}
 \begin{footnotesize}

 \end{footnotesize}

 \section{Introduction}
Let the edges of the complete graph be given independent random edge weights $w(e)$ and a random cost $c(e)$ for $e\in E(K_n)$. We are interested in the problem of estimating the minimum weight of a combinatorial structure $S$ where the total cost of $S$ is bounded by some value $C$.   More generally, we allow $r$ costs $\bc(e)=(c_i(e),i=1,2,\ldots,r)$ for each edge.  The distribution of weights $w(e)$ will be independent copies of $Z_E^\a$ where $Z_E$ denotes the exponential rate one random variable and $\a\leq 1$. The distribution of costs $c_i(e)$ will be $Z_E^\b$ where $\b\leq 1$.  In Section \ref{mgd} we will see that, since we are allowing powers of exponentials, a simple coupling argument will allow us to model a very general class of independent weights and costs, where we require just that the densities satisfy $f(x)\approx cx^\g,\g\geq 0$ as $x\to 0$; here we mean that $cx^\g/f(x)\to 1$ as $x\to 0$.

Suppose we are given cost budgets of $\bC=(C_i,i=1,2,\ldots,r)$ and we consider the following problem: let $\cS$ denote some collection of combinatorial strutures such as paths, matchings, Hamilton cycles, we would like to solve
\[
Opt(\cS,\bC):\text{Minimise }w(S)\text{ subject to }S\in \cS\text{ and }c_i(S)\leq C_i,i=1,2,\ldots,r,
\]
and let 
\[
\text{$w^*(\bC)$ denote the minimum value in $Opt(\cS,\bC)$.}
\]

We remark that Frieze and Tkocz \cite{FT1}, \cite{FT2} have considered finding minimum weight spanning trees or arborescences in the context of a single cost constraint and uniform $[0,1]$ weights and costs. I.e. the case where $\cS$ is the set of spanning trees of $K_n$ and the case where $\cS$ is the set of spanning arborescences of $\vec{K}_n$. In these cases, they obtain asymptotically optimal estimates for those problems, whereas for the problems in the present paper we have only obtained estimates that are correct to within a constant factor.

The first problem we study involves paths, here denoted as minimum weight paths for consistency with the remainder of the paper. Let $\cP(i,j)$ denote the set of paths from vertex $i$ to vertex $j$ in $K_n$.

{\bf Constrained Minimum Weight Path (CMWP)}: $Opt(\cP(1,n),\bC)$.\\
Without the constraint $\bc(P)\leq \bC$, there is a beautiful result of Janson \cite{Jan} that gives a precise value for the expected minimum weight of a path, when the $w(e)$'s are independent exponential mean one. With the constraints, we are only able to estimate the expected minimum weight up to a constant (but can do so for a more general class of distributions). 

Throughout the paper we let
\[
\D=\prod_{i=1}^rC_i
\]
be the product of the cost thresholds.  Our results show that for the structures we consider, this product of the cost thresholds controls the dependency of the minimum-weight structure on the vector of cost constraints.  In particular, for the minimum weight path problem, we have: 
\begin{theorem}\label{th1}
If $\frac{n{\D^{1/\b}}}{\log^{r/\b}n}\to\infty$ and $C_i\leq 10\log n,\,i=1,2,\ldots,r$ then w.h.p.
\[
w^*(\bC)=\Theta\bfrac{\log^{r\a/\b+1}n}{n^\a{\D^{\a/\b}}}.
\]
$A=\Theta(B)$ denotes $A=O(B)$ and $B=O(A)$. And here, the hidden constants depend only on $r,\a,\b$.
\end{theorem}
For the unconstrained problem, see Hassin and Zemel \cite{HZ}, Janson \cite{Jan} and Bhamidi and van der Hofstad \cite{BH}, \cite{BH1}.

Now consider the case of perfect matchings in the complete bipartite graph $K_{n,n}$. Let $\cM_2$ denote the set of perfect matchings in $K_{n,n}$.


{\bf Constrained Assigment Problem (CAP)}: $Opt(\cM_2,\bC)$.
\begin{theorem}\label{th2}
If $\D^{1/\b}\gg n^{r/\b-1}\log n$\footnote{Here $A=A(n)\gg B=B(n)$ if $A/B\to \infty$ as $n\to\infty$.} and $C_i\leq n,\,i=1,2,\ldots,r$ then w.h.p.
\[
w^*(\bC)=\Theta\bfrac{n^{1+r\a/\b-\a}}{{\D^{\a/\b}}}.
\]
\end{theorem}
We note that requiring a lower bound on $\D$ is necessary in Theorem \ref{th2}. Indeed, if ${\D^{1/\b}}\leq e^{-(r+1)}\b n^{r/\b-1}$ then the optimization problem is infeasible w.h.p. To see this we bound the expected number of feasible solutions as follows: let $Z_1,Z_2,\ldots,Z_r$ be independent sums of $n$ independent copies of $Z_E^{1/\b}$. Then,
\[
n!\prod_{i=1}^r\Pr\brac{Z_i\leq C_i}\leq n!\prod_{i=1}^r\frac{C_i^{n/\b}}{\b^nn!n^{n(1/\b-1)}}\leq \bfrac{e^r\D^{1/\b}}{\b n^{r/\b-1}}^n=o(1).
\]
We use Lemma \ref{lemB} here to bound $\Pr(Z_i\leq C_i)$. We note that a problem similar to this was studied by Arora, Frieze and Kaplan \cite{AFK} with respect to the worst-case.

Now consider the case of perfect matchings in the complete graph $K_{n}$. Let $\cM_1$ denote the set of perfect matchings in $K_{n}$.

{\bf Constrained Matching Problem (CMP)} $Opt(\cM_1,\bC)$.
\begin{theorem}\label{th3}
If  ${\D^{1/\b}}\gg n^{r-1}\log n$  and $C_i\leq n,\,i=1,2,\ldots,r$ then w.h.p.
\[
w^*(\bC)=\Theta\bfrac{n^{1+r\a/\b-\a}}{{\D^{\a/\b}}}.
\]
\end{theorem}

Now consider the Travelling Salesperson Problem (TSP). Let $\cT$ denote the set of Hamilton cycles in $K_n$ or the set of directed Hamilton cycles in $\vec{K}_n$.,

{\bf Constrained Travelling Salesperson Problem (CSTSP)} $Opt(\cT,\bC)$.
\begin{theorem}\label{th4}
If ${\D^{1/\b}}\gg n^{r-1}\log n$ and  and $C_i\leq n,\,i=1,2,\ldots,r$ then w.h.p.
\[
w^*(\bC)=\Theta\bfrac{n^{1+r\a/\b-\a}}{{\D^{\a/\b}}}.
\]
\end{theorem}
\section{Structure of the paper}
We prove the above theorems in their order of statement. The upper bounds are proved as follows: we consider the random graph $G_{n,p}$ (or bipartite graph $G_{n,n,p}$ or digraph $D_{n,p}$) for  suitably chosen $p$ associated with the random costs. We then seek minimum weight objects contained in these random graphs. The definition of $p$ is such that objects, if they exist, automatically satisfy the cost constraints. For minimum weight paths we adapt the methodology of \cite{Jan}. For the remaining problems we use theorems in the literature stating the high probability existence of the required objects when each vertex independently chooses a few (close) random neighbors.

In Section \ref{mgd} we consider more general distributions. We are able to extend the above theorems under some extra assumptions about the $C_i$.
\section{CSP}
\subsection{Upper Bound for CSP}
In the proof of the upper bound, we first consider weights $\hw(e)$ where the $\hw(e)$ are independent exponential mean one random variables. The costs will remain independent copies of $Z_E^\b$. We will then use Holder's inequality to obtain the final result.
\subsection{$\frac{\log^{2+r/\b}n}{n}\leq {\D^{1/\b}}$ and $C_i\leq 10\log n,\,i=1,2,\ldots,r$}\label{upper}
Suppose now that we let $L=10\log n$ and 
\[
E_0=\set{e:c_i(e)\leq \frac{C_i}{L},i=1,2,\ldots.r}.
\]
The proof in this case goes as follows:
\begin{enumerate}[(i)]
\item We search for short paths that only use edges in $E_0$ and note that the graph $([n],E_0)$ is distributed as $G_{n,p}$.
\item Observe that any path using fewer than $L$ edges of $E_0$ automatically satifies the cost constraints.
\item A simple calculation shows that w.h.p. the number of edges between a set $S$ of size $k$ and the remaining vertices is close to the expectation $k(n-k)p$ for all sets of vertices $S$, see \eqref{defES} and \eqref{cE}.
\item We run Dijkstra's algorithm for finding shortest (now minimum weight) paths from vertex 1. We use Janson's argument \cite{Jan} to bound the distance to the $m=n/3$ closest vertices $V_1$. We need the claim in item (iii) here.
\item We repeat (iv), starting from vertex set $n$, to obtain the $m=n/3$ closest vertices $V_2$. If $V_1\cap V_2\neq \emptyset$ we will have found a path of low enough weight, otherwise we calim that w.h.p. there will be a low enough weight edge joining $V_1,V_2$.
\item We then argue that the trees constructed by the Dijkstra algorithm are close to being Random Recursive Trees and we can easily bound their height. Showing that we can use item (ii).
\item We finally use Holder's inequality to switch from $\hw$ to $w$.
\end{enumerate}
We first bound the value of $p$.
\beq{defp}{
p=\Pr\brac{e\in E_0}=\prod_{i=1}^r\brac{1-\exp\set{-\bfrac{C_i}{3 L}^{1/\b}}},
}
where $e$ is an arbitrary edge. 

We note that if $0<x\leq 1$ then $x/2\leq 1-e^{-x}\leq x$. This implies that
\beq{defpp}{
\frac{\D^{1/\b}}{2^r(3L)^{r/\b}}\leq  p \leq \frac{\D^{1/\b}}{(3L)^{r/\b}}.
}
We consider the random graph $G_{n,p}$ where edges have weight given by $\hw$ and costs $c_i(e)\leq C_i/3L,i=1,2,\ldots,r$. We modify Janson's argument \cite{Jan}.

We now deal with item (iii). We observe that w.h.p. for every set $S$ of size $k$, $e(S:\bar{S})\approx k(n-k)p$ where $e(S:T)$ is the number of edges $\set{v,w}$ with one end in $S$ and the other in $T$. We only need to check the claim for $|S|\leq n/2$. Let $\e=\frac{1}{\log^{1/3}n}$ and 
\beq{defES}{
\cE_S=\set{e(S:\bar{S})\notin (1\pm\e)k(n-k)p}\text{ and }\cE=\bigcup_{|S|\leq n/2}\cE_S. 
}
Then, using the Chernoff bounds for the binomial distribution,
\beq{cE}{
\begin{split}
\Pr(\cE)&\leq \sum_{k=1}^{n/2}\binom{n}{k}\Pr(Bin(k(n-k),p)\notin (1\pm\e)k(n-k)p)\\
&\leq 2\sum_{k=1}^{n/2}\bfrac{ne}{k}^ke^{-\e^2k(n-k)p/3}\\
&=2\sum_{k=1}^{n/2}\bfrac{ne^{1-\e^2(n-k)p/3}}{k}^k\\
&\leq 2\sum_{k=1}^{n/2}\bfrac{ne^{-\Omega(\log^{4/3}n)}}{k}=o(1).
\end{split}
}
We now continue with item (iv). We set $S_1=\set{1}$ and $d_1=0$ and consider running Dijkstra's algorithm \cite{Dijk}. At the end of Step $k$ we will have computed $S_k=\set{1=v_1,v_2,\ldots,v_k}$ and $0=d_1,d_2,\ldots,d_k$ where $d_i$ is the minimum weight of a path from 1 to $i,i=1,2,\ldots,k$. Let there be $\n_k$ edges from $S_k$ to $[n]\setminus S_k$. Arguing as in \cite{Jan} we see that $d_{k+1}-d_k=Z_k$ where $Z_k$ is the minimum of $\n_k$ independent exponential mean one random variables. Also, the memoryless property of the exponential distribution implies that $Z_k$ is independent of $d_k$. It follows that for $k<n/2$,
\mult{meank0}{
\E(d_k\mid\neg\cE)=\E\brac{\sum_{i=1}^k\frac{1}{\n_i}\bigg|\neg\cE} =\sum_{i=1}^k\frac{1+o(1)}{i(n-i)p} =\frac{1+o(1)}{np}\sum_{i=1}^k\brac{\frac{1}{i}+\frac{1}{n-i}}\\
=\frac{1+o(1)}{np}\brac{H_k+H_{n-1}-H_{n-k+1}},
}
where $H_k=\sum_{i=1}^k\frac{1}i$.

By the same token,
\beq{vark0}{
\Var(d_k\mid\neg\cE)=\sum_{i=1}^k\Var(Z_i\mid\neg \cE)= \sum_{i=1}^k\frac{1+o(1)}{(i(n-i)p)^2}=O((np)^{-2}).
}
We only pursue the use of Dijkstra's algoritm from vertex 1 for $m=n/3$ iterations. It follows from \eqref{meank0} and \eqref{vark0} and the Chebyshev inequality that we have w.h.p.
\beq{dn}{
d_{m}\approx \frac{\log n}{np}. 
}
We next deal with item (vi). The tree built by Dijkstra's algorithm is (in a weak sense) close in distribution to a random recursive tree i.e. vertex $v_{k+1}$ attaches to a near uniformly random member of $\set{v_1,v_2,\ldots,v_k}$. Indeed, assuming $\cE$ does not occur, 
\[
\Pr(v_{k+1}\text{ attaches to }v_i)=\frac{e(v_i:\bar{S}_k)}{\n_k}\leq \frac{(1+\e)(n-1)p}{(1-\e)k(n-k)p}.
\]
Hence, if $T$ is the tree constructed in the first $m$ rounds of Dijkstra's algorithm, then 
\beq{height}{
\begin{split}
\Pr(height(T)\geq L)&\leq \sum_{1<t_1<\cdots<t_L<m}\prod_{i=1}^L\frac{3(1+\e)}{2(1-\e)t_i}\\
&\leq \frac{1}{L!}\bfrac{3(1+\e)}{2(1-\e)}^L\brac{\sum_{i=1}^n\frac{1}{i}}^L\\
&\leq \bfrac{3(\log n+1)e^{1+o(1)}}{2L}^L=o(1).
\end{split}
}
It follows from \eqref{defpp},  \eqref{dn} and \eqref{height}  that w.h.p., for every $v\in V_1=S_{m}$,  there exists a path $P$ from 1 to $v$ of weight at most 
\[
\l\approx\l_0=\frac{\log n}{np}\lesssim\frac{30^{r/\b}\log^{r/\b+1}n}{n\Upsilon^{1/\b}}
\]
and costs $c_i(P)\leq LC_i/3L\leq C_i/3$.\footnote{Here we write $A=A(n)\lesssim B=B(n)$ if $A\leq (1+o(1))B$.}

We now deal with item (v). We next consider applying Dijkstra's algorithm to find a minimum weight path from vertex $n$ to other vertices. Using the same argument as above, we see that we can find $m$ vertices $V_2$ that are within distance $\l_0$ of vertex $n$. If $V_1\cap V_2\neq\emptyset$ then we have found a path of weight at most $2\l_0$ between vertex 1 and vertex $n$.

If $V_1,V_2$ are disjoint then w.h.p. there is an edge of weight $20/np$ between them. Indeed,
\[
\Pr(\exists V_1,V_2\text{ with no such edge})\leq \binom{n}{m}^2(e^{-20/np})^{n^2/9}=o(1).
\]
This yields a path $P$ with
\begin{align}
\hw(P)&\leq 2\l_0+\frac{20}{np}\leq\frac{3\cdot 30^{r/\b}\log^{r/\b+1}n}{n\Upsilon^{1/\b}}.\label{eq1}\\
c_i(P)&\leq \frac{2C_i}{3}+\frac{C_i}{3}= C_i,\quad i=1,2,\ldots,r.\label{eq2}
\end{align}
(Here we have used $C_i\geq {\D^{1/\b}}/L^{r-1}\gg p$.)

We now deal with item (vii). We use Holder's inequality to yield
\beq{Hold1}{
w(P)=\sum_{e\in P}\hw(e)^\a\leq \brac{\sum_{e\in P}\hw(e)}^\a L^{1-\a}= O\bfrac{\log^{r\a/\b+1}n}{n^\a\D^{\a/\b}}.
}
This completes the proof of Theorem \ref{th1} for this case. 
\subsection{$\frac{3\om\log^{r/\b}n}{n}\leq {\D^{1/\b}}\leq \frac{\log^{r/\b+2} n}{n}$  and $C_i\leq 10\log n,\,i=1,2,\ldots,r$}\label{smalldelta}
The proof is similar to that of Section \ref{upper}, but requires some changes in places. The problem is that we cannot now assume the non-occurrence of $\cE$. Other than this, the proof will follow the same strategy. Our problem therefore is to argue that w.h.p. $e(S_k:\bar{S}_k)$ is sufficiently large.
\begin{enumerate}[(a)]
\item We now have to keep track of the size of  $e(S_k:\bar{S}_k)$ as a random process. This is equation \eqref{l0}.
\item The term $\eta_k$ is the number of edges between $v\notin S_k$ and $S_k$. We don't want this to be large, as it reduces $e(S_{k+1}:\bar{S}_{k+1})$. So, we do not add vertices to $S_k$ if $\eta_k\geq 2np$, which only happends rarely.
\item Finally, we have to work harder in the case where $V_1,V_2$ are disjoint. We need to use edges of slightly higher cost in order to get a low weight edge in $e(V_1:V_2)$.
\end{enumerate}

Let $p$ be as in \eqref{defp} where $L=20\log n$. Note that from \eqref{defpp} we see that
\[
p\leq \frac{\log^2n}{n}.
\]

We again consider the random graph $G_{n,p}$ where edges have weight given by $\hw$ and costs at most $C_i/3L$ and again modify Janson's argument \cite{Jan}. We also restrict our search for paths, avoiding vertices of high degree.

We set $S_1=\set{1}$ and $d_1=0$. At the end of Step $k$ we will have computed $S_k=\set{1=v_1,v_2,\ldots,v_k}$ and $0=d_1,d_2,\ldots,d_k$ where $d_i$ is the minimum weight of a path  from 1 to $i,i=1,2,\ldots,k$. Let there be $\n_k$ edges from $S_k$ to $[n]\setminus S_k$. We cannot rely on $\cE$ of \eqref{cE} not to occur and so we need to modify the argument here. 

{\bf Assumption: $1\leq k\leq n_0=1/3p$}\\
{\bf Modification:} if our initial choice $v$ for $v_{k+1}$ satisfies $e(v:\bar{S}_k)\geq 2np$ then we reject $v$ permanently from the construction of paths from vertex 1.

The initial aim is roughly the same, we want to show that w.h.p. 
\beq{l0}{
\sum_{\ell\leq k}\n_\ell\geq (1-o(1))knp.
}
For $v\notin S_k$, let $\eta_{k,v}=e(S_k:\set{v})$ and $\eta_k=\eta_{k,v_{k+1}}$. Then, w.h.p.
\beq{nuk}{
\n_{k+1}\geq \n_k-\eta_{k}+B_k\text{ where }B_k=Bin(n_1,p)1_{Bin(n,p)\leq 2np},
}
where $n_1=n-2n_0$.

The binomials are independent here. This is because the edges between $v_{k+1}$ and $\bar{S}_k$ have not been exposed by the algorithm to this point. The number of trials $n_1$ comes from the following: we know from the Chernoff bounds  that
\beq{pq}{
\Pr(Bin(n,p)\geq 2np)\leq e^{-np/3}.
}
It follows from the Markov inequality that w.h.p. there are at most $ne^{-np/4}$ instances where the modification is invoked. This means that w.h.p. the initial choice for $v_k$ has at least $n-n_0-ne^{-np/4}\geq n_1$ possible neighbors. We now define
\[
S_k=\sum_{\ell=1}^kB_k.
\]
We need a lower bound for $B_k$ and an upper bound for $\eta_k$.  We next observe that if 
\[
\e=(np)^{-1/3}
\] 
then
\beq{l1}{
\Pr(B_k\leq (1-\e)np)=\Pr(Bin(n_1,p)\geq 2np)+\Pr(Bin(n_1\leq (1-\e)np))\leq (1+o(1))e^{-\e^2np/3}.
}
It follows that if $k_0=\min\set{n_0,e^{\e^2np/4}}$ then w.h.p.
\beq{l2}{
\Pr(\exists 0\leq k\leq k_0: B_k\leq (1-\e)np)\leq (1+o(1))k_0e^{-\e^2np/3}\leq e^{-\e^2np/12}.
}
For $k\geq k_0$, we use the fact that $S_k$ is the sum of bounded random variables. Hoeffding's inequality \cite{Hoef} gives that 
\[
\Pr(S_k\leq \E(S_k)-t)\leq \exp\set{-\frac{2t^2}{4kn^2p^2}}.
\]
Now $\E(B_k)\geq (1-\e)np$ and so putting $t=k^{2/3}np$ we see that 
\[
\Pr(S_k\leq (1-\e)knp-k^{2/3}np)\leq e^{-k^{1/3}/2}.
\]
So
\beq{l3}{
  \Pr(\exists k\geq k_0:\;S_k\leq (1-\e)knp-k^{2/3}np)\leq \sum_{k\geq k_0}e^{-k^{1/3}/2}=o(1).
}

We next observe that 
\begin{align}
\Pr(\exists S:\;|S|=s\leq 1/3p,e(S:S)\geq s+r)&\leq \sum_{s=1}^{1/3p} \binom{n}{s} \binom{s(s-1)/2}{s+r}p^{s+r}\nonumber\\
&\leq \sum_{s=1}^{1/3p} \bfrac{e^2np}{2}^s \bfrac{sep}{2}^r.\label{Eq2}
\end{align}
Putting $r=s(np)^{1/2}$, the RHS of \eqref{Eq2} becomes
\[
\sum_{s=1}^{1/3p}\brac{\frac{e^2np}{2}\bfrac{sep}{2}^{(np)^{1/2}}}^s\leq \sum_{s=1}^{1/3p}\brac{\frac{e^2np}{2}\bfrac{e}{6}^{(np)^{1/2}}}^s=o(1).
\]
It follows that w.h.p.,
\beq{tt}{
\sum_{\ell=1}^k\eta_\ell=e(S_k)\leq 2((np)^{1/2}+1)k.
}
It then follows from  \eqref{nuk} and \eqref{l2} and \eqref{l3} and  \eqref{tt} that w.h.p.
\beq{case1}{
\n_k\geq (1-o(1))knp-2((np)^{1/2}+1)k\geq (1-o(1))knp.
}
Arguing as in \cite{Jan} we see that $d_{k+1}-d_k=Z_k$ where $Z_k$ is the minimum of $\n_k$ independent exponential mean one random variables. Also, $Z_k$ is independent of $d_k$. It follows that for $k<n$,
\beq{meank}{
\E(d_k)=\E\brac{\sum_{i=1}^k\frac{1}{\n_i}} \leq \sum_{i=1}^k\frac{1+o(1)}{inp} =\frac{1+o(1)}{np}\sum_{i=1}^k\frac{1}{i}=\frac{1+o(1)}{np} H_k,
}
where $H_k=\sum_{i=1}^k\frac{1}i$.

By the same token,
\beq{vark}{
\Var(d_k)=\sum_{i=1}^k\Var(Z_i)=\sum_{i=1}^k\frac{1+o(1)}{(inp)^2}=O((np)^{-2}).
}
It follows from \eqref{meank} and \eqref{vark} and the Chebyshev inequality that w.h.p. we have
 $d_{n_0}\lesssim\frac{\log n}{np}$. Let $V_1$ denote the $n_0$ vertices at this distance from vertex 1.

We next consider applying Dijkstra's algorithm to find a minimum weight path from vertex $n$ to other vertices. Using the same argument as above, we see that we can find $n_0$ vertices $V_2$ that are within distance $\frac{(1+o(1))\log n}{np}$ of vertex $n$. If $V_1\cap V_2\neq\emptyset$ then we have found a path of weight at most $\frac{(2+o(1))\log n}{np}$ between vertex 1 and vertex $n$.

If $V_1,V_2$ are disjoint then we will use the edges 
\[
E_1=\set{e:c_i(e)\in\left[\frac{C_i}{L},\frac{2C_i}{L}\right],i=1,2,\ldots.r}.
\]
Given $e=\set{x,y}\in V_1:V_2$, then given the history of Dijkstra's algorithm so far, either $e\in E_0$ or 
we can say that
\beq{E1}{
\Pr(e\in E_1\mid e\notin E_0)\geq \Pr(e\in E_1)=(1-e^{-(2^{1/\b}-1)p^{1/\b}})^r.
}
For the equation in \eqref{E1} we use
\mult{p2p}{
\Pr(p\leq Z_E^\b\leq 2p)=\Pr(Z_E^\b\geq p)(1-\Pr(Z_E^\b\geq 2p\mid Z_E\geq p))=\\
e^{-p^{1/\b}}\brac{1-\frac{\Pr(Z_E^\b\geq 2p)}{\Pr(Z_e^\b\geq p)}}=e^{-p^{1/\b}}\brac{1-e^{-(2^{1/\b}-1)p^{1/\b}}}=e^{-p^{1/\b}}-e^{-(2p)^{1/\b}}\geq \frac{(2^{1/\b}-1)p^{1/\b}}{2}.
}
For the inequality in \eqref{p2p} we use the fact that we now have $p\leq \frac{\log^2n}{n}$.

Then we search for an edge in $E_2=\set{e\in E_1:\hw(e)\leq 1/np}$. And, 
\[
\Pr(E_2\cap(V_1:V_2)=\emptyset)\leq \brac{1-\brac{\frac{(2^{1/\b}-1)p^{1/\b}}{2}}(1-e^{-1/np})}^{1/9p^2}\\
\leq \brac{1-\frac{(2^{1/\b}-1)p^{1/\b}}{2np}}^{1/9p^2}=o(1).
\]
This yields a path of weight at most $\frac{(2+o(1))\log n}{np}+\frac{1}{np}=\frac{(2+o(1))\log n}{np}$. 

We deal with the height of the Dijkstra trees. Let $T$ be the tree constructed by Dijkstra's algorithm and let $\xi_i,i\leq k$ denote the number of edges from $v_i$ to $V_1\setminus S_i$.
\begin{align*}
\Pr(height(T)\geq L)&\leq \E\brac{\sum_{1<t_1<\cdots<t_L<n_0}\prod_{i=1}^L\frac{\xi_{t_i}}{\n_{t_{i+1}-1}}}\\
&\leq \E\brac{\sum_{1<t_1<\cdots<t_L<n_0}\prod_{i=1}^L\frac{2np}{\n_{t_{i+1}-1}}}\\
&\leq \E\brac{\frac{1}{L!}\brac{\sum_{i=1}^{n_0}\frac{2np}{\n_i}}^L}\\
&\leq o(1)+\frac{(2enp)^L}{(np)^LL!}\brac{\sum_{i=1}^{n_0}\frac{1+o(1)}{i}}^L\\
\noalign{\text{The first $o(1)$ term here is the probability that there is a small $\n_k$ and this is covered by \eqref{case1}.}}\\
&=o(1),
\end{align*}
since $L\geq 20\log n$.

It follows from the above that w.h.p. there exists a path $P$
\beq{final}{
\text{where } \hw(P)\lesssim \frac{2\log n}{np}\text{ and }c_i(P)\leq \frac{(2L+2)C_i}{3L}<C_i,i=1,2,\ldots,r.
}
Arguing as for \eqref{Hold1} we see that
\beq{Hold2}{
w(P)\leq \hw(P)^\a L^{1-\a}= O\bfrac{\log^{r\a/\b+1}n}{n^\a\D^{\a/\b}}.
}
\subsection{Lower Bound for CSP}\label{lowcsp}
This is a straightforward use of the first moment method. Suppose that  
\[
{\D^{1/\b}}=\frac{\om\log^{r/\b}n}{n},\quad L=\frac{\e\log^{r\a/\b+1}n}{n^\a\D^{\a/\b}},
\] 
where 
\[
\e=\brac{\a\b^r e^{-2}\brac{10\brac{\frac{r}{\b}+\frac{1}{\a}}}^{-(r/\b+1/\a)}}^\a,
\]
then 
\begin{align}
&\Pr\brac{\exists P:w(P)\leq L,c_i(P)\leq C_i,i=1,2,\ldots,r}\nonumber\\
&\leq  \sum_{k=1}^{n-2}n^{k-1} \bfrac{L^{k/\a}}{\a^kk!k^{k(1/\a-1)}}\prod_{i=1}^r\frac{C_i^{k/\b}}{\b^kk!k^{k(1/\b-1)}} \label{lowexp}\\
&\leq \frac{1}{n}\sum_{k=1}^{n-1}\brac{n\cdot\frac{e\e^{1/\a}\log^{r/\b+1/\a}n}{\a nk^{1/\a}\D^{1/\b}}\cdot \frac{e\D^{1/\b}}{\b^r k^{r/\b}}}^k\nonumber\\
&=\frac{1}{n}\sum_{k=1}^{n-1}\bfrac{e^2\e^{1/\a}\log^{r/\b+1/\a}n}{\a\b^rk^{r/\b+1/\a}}^k\nonumber\\
&=\frac{1}{n}\sum_{k=1}^{\frac12\log n}\bfrac{e^2\e^{1/\a}\log^{r/\b+1/\a}n}{\a\b^rk^{r/\b+1/\a}}^k+ \frac{1}{n}\sum_{k=\frac12\log n}^{n-1}\bfrac{e^2\e^{1/\a}\log^{r/\b+1/\a}n}{\a\b^rk^{r/\b+1/\a}}^k\nonumber\\
&\leq \frac{1}{n}\sum_{k=1}^{\frac12\log n}\bfrac{\log n}{10(r/\b+1/\a)k}^{(r/\b+1/\a)k}+ \frac{1}{n}\sum_{k=\frac12\log n}^{n-1}10^{-k}\nonumber\\
&=o(1).\nonumber
\end{align}
{\bf Explanation for \eqref{lowexp}}: we choose a path of length $k$ from 1 to $n$ in at most $n^{k-1}$ ways.  Then we use Lemma \ref{lemB} $r+1$ times. Then we use the union bound. 
\section{Upper Bounds}
\subsection{Upper Bound for CAP}\label{CAP}

Let $G$ denote the subgraph of $K_{n,n}$ induced by the edges that satisfy $c_i(e)\leq C_i/n$ for $i=1,2,\ldots,r$. Let 
\[
p=\Pr\brac{c_i(e)\leq \frac{C_i}{n},i=1,2,\ldots,r}=\prod_{i=1}^r\brac{1-\exp\set{-\bfrac{C_i}{n}^{1/\b}}}.
\] 
and note that 
\[
\frac{\log n}{n}\ll \frac{\D^{1/\b}}{2^rn^{r/\b}}\leq p\leq \frac{\D^{1/\b}}{n^{r/\b}}.
\]
The approach for this and the remining problems is
\begin{enumerate}[(i)]
\item Look for a small weight structure in an edge weighted random graph $G$. In this case the random bipartite graph $G_{n,n,p}$.
\item Use an idea of Walkup \cite{W1} to construct a random subgraph $H$ of $G$ that only uses edges of low weight.
\item Use a result from the literature that states that w.h.p. the edges of $H$ contain a copy of the desired structure.
\end{enumerate}
$G$ is distributed as $G=G_{n,n,p}$. Note that by construction, a perfect matching $M$ of $G$ satisfies $c_i(M)\leq C_i,i=1,2,\ldots,r$. 

Let $d=np$ and note that because $dnp\gg\log n$ the Chernoff bounds imply that w.h.p. every vertex has degree $\approx d$. Now each edge of $G$ has a weight uniform in $[0,1]$. Following Walkup \cite{W1} we replace $w(e),e=(x,y)$ by $\min\set{Z_1(e),Z_2(e)}$ where 
\beq{defZ1}{
\text{$Z_1,Z_2$ are independent copies of $Z_W$ where $\Pr(Z_W\geq x)^2=\Pr(Z_E^\a\geq x)$.}
} 
We assign $Z_1(e)$ to $x$ and $Z_2(e)$ to $y$. 

Let $X,Y$ denote the bipartition of the vertices of $G$. Now consider the random bipartite graph $H$ where each $x\in X$ is incident to the two $Z_1$-smallest edges incident with $x$. Similarly, $y\in Y$ is incident to the two $Z_2$-smallest edges incident with $y$. Walkup \cite{W2} showed that $H$ has a perfect matching w.h.p.  The expected weight of this matching is asymptotically at most
\beq{capexp}{
\bfrac{2^\a n}{d^\a}\brac{\G\brac{1+\frac{1}{\a}}+\G\brac{2+\frac{1}{\a}}}\times \frac12= O\bfrac{n^{1+r\a/\b-\a}}{\D^{\a/\b}}.
}
This follows from (i) the the expression given in Corollary \ref{cor1} for the expected minimum and second minimum of $d$ copies of $Z$ and (ii) the matching promised in \cite{W2} is equally likely to select a minimum or a second minimum weight edge.

The selected matching is the sum of independent random variables with exponential tails and so will be concentrated around its mean.
\subsection{Upper Bound for CMP}
We let $p,d$ be as in Section \ref{CAP}. We replace Walkup's result \cite{W2} by Frieze's result \cite{F2out} that the random graph $G_{2-out}$ contains a perfect matching w.h.p. The random graph $G_{k-out}$ has vertex set $[n]$ and each vertex $v\in [n]$ independently chooses $k$ random edges incident with $v$. We again replace $c(e),e=(x,y)$ by $\min\set{Z_1(e),Z_2(e)}$ where $Z_1,Z_2$ are independent copies of $Z_W$ and associate one copy with each endpoint of the edge. We consider the random graph $H$ where each $v\in [n]$ is incident to the two $Z_W$-smallest edges incident with $x$. This is distributed as $G_{2-out}$ and we obtain an expression similar to that in \eqref{capexp}.

We have concentration around the mean as in Section \ref{CAP}.
\subsection{Upper bound for CSTSP/CATSP}
For the symmetric case we replace $w(e),e=\set{x,y}$ by $\min\set{Z_1(e),Z_2(e)}$ for each edge of $K_n$ and for the asymmetric case we replace $w(e),e=(x,y)$ by $\min\set{Z_1(e),Z_2(e)}$ for each directed edge of $\vec{K}_n$. In both cases we associate one copy of $Z_W$ to each endpoint of $e$. We define $p,d$ as in Section \ref{CAP} and consider either the random graph $G_{n,p}$ or the random digraph $D_{n,p}$.

For the symmetric case, we consider the random graph $H$ that includes the 3 cheapest edges associated with each vertex, cheapest with respect to $Z_W(e)$. This will be distributed as $G_{3-out}$ which was shown to be Hamiltonian w.h.p. by Bohman and Frieze \cite{BF}. For the asymmetric case, we consider the random digraph $H$ that includes the 2 cheapest out-edges and the 2 cheapest in edges associated with each vertex, cheapest with respect to $Z_W(e)$. This will be distributed as $D_{2-in,2-out}$ which has vertex set $[n]$ and where each vertex $v$ independently chooses 2 out- and in-neighbors. The random digraph $D_{2-in,2-out}$ was shown to be Hamiltonian w.h.p. by Cooper and Frieze \cite{CF}. 

The expected weight of the tour promised by \cite{BF} or by \cite{CF} is asymptotically $O(n^{1+r\a/\b-\a}/\D^{\a/\b})$ as in Section \ref{CAP}. We have concentration around the mean as in Section \ref{CAP}.

\section{Lower Bounds}
We proceed as in Section \ref{lowcsp}. Suppose that $\D=\om n^{r/\b-1}\log n$ and $L=\frac{\e n^{1+r\a/\b-\a}}{\D^{\a/\b}}$ where $\e$ will be a sufficnetly small constant. Let $\Lambda$ denote the relevant structure, matching or cycle. Then, by the union bound and Lemma \ref{lemB}, we have for CAP,CSTSP,CATSP,
\multstar{
\Pr\brac{\exists \Lambda : w(\Lambda )\leq L\text{ and }c_i(\Lambda )\leq C_i,i=1,2,\ldots,r}\leq 
 n!\cdot\frac{L^{n/\a}}{\a^nn!n^{n(1/\a-1)}}\cdot \prod_{i=1}^r\frac{C_i^{n/\b}}{\b^nn!n^{n(1/\b-1)}}\\
\leq \brac{\frac{\e^{1/\a}n^{1/\a+r/\b-1}}{\a n^{1/\a-1}\D^{1/\b}}\cdot \frac{e^r\D^{1/\b}}{\b n^{r/\b}}}^n=o(1),
}
for $\e$ sufficiently small.

For CMP, assuming that $n=2m$,
\multstar{
\Pr\brac{\exists \Lambda : w(\Lambda )\leq L\text{ and }c_i(\Lambda )\leq C_i,i=1,2,\ldots,r}\leq 
 \frac{n!}{m!2^m}\cdot\frac{L^{m/\a}}{\a^mm!m^{m(1/\a-1)}}\cdot \prod_{i=1}^r\frac{C_i^{m/\b}}{\b^mm!m^{m(1/\b-1)}}\\
\leq \brac{\frac{\e^{1/\a}m^{1/\a+r/\b-1}}{2\a m^{1/\a-1}\D^{1/\b}}\cdot \frac{e^r\D^{1/\b}}{\b m^{r/\b}}}^m=o(1),
}
for $\e$ sufficiently small.

\section{More general distributions}\label{mgd}
We follow an argument from Janson \cite{Jan}. We will asssume that $w(e),$ has the distribution function $F_w(t)=\Pr(X\leq t)$, of a random variable $X$, that satisfies $F_w(t)\approx at^{1/\a},\a\leq 1$ as $t\to 0$. For the costs $c_i(e)$ we have $F_c(t)\approx bt^{1/\b},\b\leq 1$. The constants $a,b>0$ can be dealt with by scaling and so we assume that $a=b=1$ here. For a fixed edge and say, $w(e)$, we consider random variables $w_<(e),w_>(e)$ such that $w_<(e)$ is distributed as $Z_E^{\a+\e_n}$ and $w_>(e)$ is distributed as $Z_E^{\a-\e_n}$, where $\e_n=1/10\log n$. (This choice of $\e_n$ means that $n^{\a+\e_n}=e^{1/10}n^\a$.)  Then let $U(e)$ be a uniform $[0,1]$ random variable and suppose that $X$ has the distribution $F^{-1}(U)$. We couple $X,w_<,w_>$ by generating $U(e)$ and then $w_<(e)=F_<^{-1}(U)=\log\bfrac{1}{1-u}^{\a-\e_n}$ and $F_>$ is defined similarly. The coupling ensures that $w_<(e)\leq w(e)\leq w_>(e)$ as long as $w(e)\leq\e_n$.

Given the above set up, it only remains to show that w.h.p. edges of length $w(e)>\e_n$ or cost $c_i(e)>\e_n$ are not needed for the upper bounds proved above. We can ignore the lower bounds, because they only increase if we exclude long edges. 

{\bf Assumptions for CMWP.} For the minimum weight path problem we will assume that $\D^{1/\b}\gg \frac{\log^{1+r/\b}n}{n}$, which is a $\log n$ factor larger than required for Theorem \ref{th1}. We will assume that $C_i=o(1)$ and then we only use edges of cost of order $C_i/\log n\ll \e_n$.

Observe that the minimum weight of a path from 1 to $n$ is at most $\frac{4\log n}{np}$ w.h.p. and this is less than $\e_n$ because of the assumption $\frac{\log^{1+r/\b}n}{n}\ll {\D^{1/\b}}$ and the definition of $p$ (see \eqref{defpp}). 

{\bf Assumptions for the other problems,}  We deal with costs by assuming that $C_i=o(n/\log n),i=1,2,\ldots r$. It is then a matter of showing that w.h.p. the first few order statistics of $Z_W$ are very unlikely to be greater than $\e_n$. ($Z_W$ is defined in \eqref{defZ1}.) But in all cases this can be bounded as follows: let $W_1,W_2,\ldots,W_m, m\geq n/2$ be independent copies of $Z_W$. Then,
\[
\Pr(|\set{i:W_i\leq \e_n}|\leq 3)\leq m^3(1-(1-e^{-\e_n^{1/\a}})^{1/2})^{m-3}=m^3e^{-m^{1-o(1)}}.
\] 
This bounds the probability of using a heavy edge at any one vertex and inflating by $n$ gives us the result we need.
\section{Conclusion}
We have given upper and lower bounds that hold w.h.p. for constrained versions of some classical problems in Combinatorial Optimization. They are within a constant factor of one another, unlike the situation with respect to spanning trees and arborescences, \cite{FT1}, \cite{FT2}, where the upper and lower bounds are asymptotically equal. It is a challenge to find tight bounds for the problems considered in this paper and to allow correlation between length and cost.

We have not made any claims about $\E(w^*(\bC))$ because there is always the (small) probability that the problem is infeasible. It is not difficiult to similarly bound the expectation conditional on feasibility.

\appendix

\section{Auxilliary Lemmas}\label{2min}
\begin{lemma}\label{lemA}
Let $\alpha > 0$ and let $Y_1, Y_2, \dots$ be i.i.d. copies of $Z=Z_E^\a$. For a positive integer $m$ and $1 \leq k \leq m$, let $X_m^{(k)}$ be the $k$th minimum of $Y_1, \dots, Y_m$. Then
\[
\E X_m^{(k)} = \Gamma\left(1+\a\right)\sum_{j=0}^{k-1}\sum_{i=0}^{j} \binom{m}{j}\binom{j}{i}(-1)^{i}(m+i-j)^{-\alpha}.
\]
In particular, if $k$ is a constant as $m \to \infty$, then
\[
\E X_m^{(k)} \approx \frac{1}{(k-1)!}\Gamma\left(k+\frac{1}{\alpha}\right)m^{-\alpha}.
\]
\end{lemma}
\begin{proof}
Note that
\[
\Pr(X_m^{(k)} > t) = \sum_{j=0}^{k-1} \binom{m}{j}\Pr(Y_1 \leq t)^j\Pr(Y_1 > t)^{n-j}
\]
(for the $k$th minimum to be larger than $t$, we need exactly $j$ variables to be at most $t$ and $m-j$ larger than $t$, $j = 0,1,\dots,k-1$). Integrating gives
\[
\E X_m^{(k)} = \int_0^\infty \Pr(X_m^{(k)} > t) \dd t =  \sum_{j=0}^{k-1} \binom{m}{j} \int_0^\infty \left(1 - e^{-t^\alpha}\right)^{j}e^{-(m-j)t^{\alpha}} \dd t.
\]
It remains to expand $\left(1 - e^{-t^\alpha}\right)^{j}$ and use $\int_0^\infty e^{-\lambda t^{\alpha}} \dd t = \Gamma\left(1+\alpha\right)\lambda^{-\alpha}$. The asymptotic statements follow by writing $(m+i-j)^{-\alpha} = m^{-\alpha}(1 + \frac{i-j}{m})^{-\alpha}$ and applying the binomial series.
\end{proof}
\begin{corollary}\label{cor1}
Let $\alpha > 0$ and let $\hY_1, \hY_2, \dots$ be i.i.d. copies of $Z_1$, where $Z_1$ is as defined in \eqref{defZ1}. For a positive integer $m$ and $1 \leq k \leq m$, let $\hX_m^{(k)}$ be the $k$th minimum of $\hY_1, \dots, \hY_m$. Then
\[
\E \hX_m^{(k)} = 2^{\a}\Gamma\left(1+\alpha\right)\sum_{j=0}^{k-1}\sum_{i=0}^{j} \binom{m}{j}\binom{j}{i}(-1)^{i}(m+i-j)^{-\alpha}.
\]
In particular, if $k$ is a constant as $m \to \infty$, then
\[
\E \hX_m^{(k)} \approx 2^{\a}\frac{1}{(k-1)!}\Gamma\left(k+\alpha\right)m^{-\alpha}.
\]
\end{corollary}
\begin{proof}
This follows from Lemma \ref{lemA} and the fact that $Z_1$ has the same distribution as $2^{\a}Z$ because we have $\Pr(Z_1 \geq x) = \Pr(Z \geq x)^{1/2} = \Pr(2^{\a}Z \geq x)$. (Here $Z=Z_E^\a$.)
\end{proof}

\begin{lemma}\label{lemB}
Let $\alpha \leq 1$ and let $Y_1, Y_2, \dots$ be i.i.d. copies of $Z_E^\a$. Then for $t \geq 0$, we have
\[
\Pr(Y_1+\dots + Y_n \leq t) \leq \frac{t^{n/\alpha}}{\alpha^nn!n^{n(1/\alpha-1)}}.
\]
\end{lemma}
\begin{proof}
Using the density,
\[
\Pr(Y_1+\dots + Y_n \leq t)  = \int_{x_1, \dots, x_n \geq 0, \sum x_i \leq t} \prod_{i=1}^n \alpha^{-1} x_i^{1/\alpha-1}e^{-x_i^{1/\alpha}} \dd x_1\dots\dd x_n.
\]
By the AM-GM inequality,
\[
\prod_{i=1}^n x_i \leq \left(\frac{\sum_{i=1}^n x_i}{n}\right)^n,
\]
and trivially $e^{-x_i^{1/\alpha}} \leq 1$, so the integrand can be pointwise bounded as follows
\[
 \prod_{i=1}^n \alpha^{-1} x_i^{1/\alpha-1}e^{-x_i^{1/\alpha}}  \leq \alpha^{-n}\left(\frac{\sum_{i=1}^n x_i}{n}\right)^{n(1/\alpha-1)} \leq \alpha^{-n} \frac{t^{n(1/\alpha-1)}}{n^{n(1/\alpha-1)}}
\]
Thus,
\[
\Pr(Y_1+\dots + Y_n \leq t)  \leq \alpha^{-n} \frac{t^{n(1/\alpha-1)}}{n^{n(1/\alpha-1)}}\cdot\mathrm{vol}\left\{x_1,\dots,x_n \geq 0, \ \sum_{i=1}^n x_i \leq t\right\} = \alpha^{-n}\frac{t^{n/\alpha}}{n!n^{n(1/\alpha-1)}}.
\]
\end{proof}


\begin{thebibliography}{99}
\bibitem{AFK} S. Arora, A.M. Frieze and H. Kaplan, A New Rounding Procedure for the Assignment Problem with Applications to Dense Graph Arrangement Problems, {\em Mathematical Programming A} 92 (2002) 1-36.
\bibitem{BH} S. Bhamidi and R. van der Hofstad, Weak disorder asymptotics in the stochastic mean-field model of distance, {\em  The Annals of Applied Probability} 22 (2010) 1907-1965.
\bibitem{BH1} S. Bhamidi and R. van der Hofstad, Weak disorder asymptotics in the stochastic mean-field model of distance II, {\em Bernoulli} 19 (2013) 363-386.
\bibitem{BF} T. Bohman and A.M. Frieze, Hamilton cycles in 3-out, {\em Random Structures and Algorithms} 35 (2009) 393-417.
\bibitem{CF} C. Cooper and A.M. Frieze, Hamilton cycles in random graphs and directed graphs, {\em Random Structures and Algorithms} 16 (2000) 369-401.
\bibitem{Dijk} E. Dijkstra, A note on two problems in connexion with graphs, {\em Numerische Mathematik} 1 (1959) 269-271. 
\bibitem{F} A.M. Frieze, On the value of a random minimum spanning tree problem, {\em Discrete Applied Mathematics} 10 (1985) 47-56.
\bibitem{F2out} A.M. Frieze, Maximum matchings in a class of random graphs, {\em Journal of Combinatorial Theory B} 40 (1986) 196-212.
\bibitem{FT1} A.M. Frieze and T. Tkocz, A randomly weighted minimum spanning tree with a random cost constraint, \href{https://arxiv.org/pdf/1905.01229.pdf}{Arxiv}.
\bibitem{FT2} A.M. Frieze and T. Tkocz, A randomly weighted minimum arborescence with a random cost constraint, \href{https://arxiv.org/pdf/1907.03375.pdf}{Arxiv}.
\bibitem{HZ} R. Hassin and E. Zemel, On shortest paths with random weights, {\em Mathematics of Operations Research} 10 (1985) 557-564.
\bibitem{Hoef} W. Hoeffding, Probability inequalities for sums of
bounded random variables, {\em Journal of the American Statistical Association} 58 (1963) 13-30.
\bibitem{Jan} S. Janson, One, two and three times $\log n/n$ for paths in a complete graph with random weights, {\em Combinatorics, Probability and Computing} 8 (1999) 347-361.
\bibitem{KS} R. Karp and M. Steele, R.M.~Karp and J.M.~Steele, Probabilistic analysis of heuristics, in The traveling salesman problem: a guided tour of combinatorial optimization, E.L.~Lawler, J.K.~Lenstra, A.H.G.~Rinnooy Kan and D.B.~Shmoys Eds. (1985) 181--206.
\bibitem{W1}  D.W. Walkup, On the expected value of a random asignment problem, {\em SIAM Journal on Computing} 8 (1979) 440-442.
\bibitem{W2} D.W. Walkup, Matchings in random regular bipartite graphs, {\em Discrete Mathematics} 31 (1980) 59-64.
\end{thebibliography}
\end{document}